\documentclass[letterpaper, 10 pt, conference]{ieeeconf}
\IEEEoverridecommandlockouts  
\usepackage{cite}
\usepackage{amsmath,amssymb,amsfonts}
\usepackage{algorithmic}
\usepackage{graphicx}
\usepackage{textcomp}
\usepackage{xcolor}
\usepackage{mathtools}

\newcommand{\G}{\mathcal{G}}
\newcommand{\V}{\mathcal{V}}

\newcommand{\BR}{\operatorname{BR}} 
\newtheorem{theorem}{Theorem}
\newtheorem{lemma}[theorem]{Lemma}

\newtheorem{example}{Example}
\newtheorem{definition}{Definition}

\title{\LARGE \bf
Convergence and consensus analysis of a class of best-response opinion dynamics
}

\author{Yuchen Xu, Yi Han$^*$, Chuanzhe Zhang, Miao Wang, Wenjun Mei
\thanks{This work was supported in part by the National Key R\&D Program of China under Grant 2022ZD0116401 and 2022ZD0116400, the National Natural Science Foundation of China under Grants 72201008 and 72131001, the Shuanghu Laboratory under Grant SH-2024JK31, Beijing Natural Science Foundation under Grant QY24049, and Peking University under The Fundamental Research Funds for the Central Universities.}
\thanks{Y. Xu, Y. Han, C. Zhang, M. Wang, and W. Mei are with the Department of Mechanics and Engineering Science, Peking University, 100871 Beijing, China.
        }%
\thanks{$^*$The corresponding author is Yi Han (han\_yi@stu.pku.edu.cn).}
}

\begin{document}

\maketitle
\thispagestyle{empty}
\pagestyle{empty}

\begin{abstract}
Opinion dynamics aims to understand how individuals’ opinions evolve through local interactions. Recently, opinion dynamics have been modeled as network games, where individuals update their opinions in order to minimize the social pressure caused by disagreeing with others. In this paper, we study a class of best response opinion dynamics introduced by Mei et al., where a parameter $\alpha > 0$ controls the marginal cost of opinion differences, bridging well-known mechanisms such as the DeGroot model ($\alpha = 2$) and the weighted-median model ($\alpha = 1$). We conduct theoretical analysis on how different values of $\alpha$ affect the system's convergence and consensus behavior. For the case when \(\alpha>1\), corresponding to increasing marginal costs, we establish the convergence of the dynamics and derive graph-theoretic conditions for consensus formation, which is proved to be similar to those in the DeGroot model. When \(\alpha<1\), we show via a counterexample that convergence is not always guaranteed, and we provide sufficient conditions for convergence and consensus. Additionally, numerical simulations on small-world networks reveal how network structure and $\alpha$ together affect opinion diversity.

\end{abstract}

\section{Introduction}

\subsection{Background and motivation}
Opinion dynamics is an interdisciplinary research area at the intersection of social science, network science, and control theory. Opinion dynamics aim to understand how individuals’ opinions evolve through social interactions. In recent years, researchers have adopted a game-theoretic perspective to construct opinion dynamics models, using cost functions to characterize the motivations behind individuals’ opinion updates. In \cite{WM-FB-GC-JH-FD:22}, Mei et al. proposed a class of opinion dynamics based on individuals’ best responses to social pressure arising from disagreement. For any individual $i$, social pressure is characterized by the following cost function: 
$$u_i(x_i) = \sum_{j=1 }^n w_{ij} \vert x_i - x_j\vert^\alpha,\quad \alpha>0.$$
Each individual updates their opinion $x_i$ by minimizing the cost function $u_i(x_i)$. Intuitively, $\alpha$ controls how individuals perceive the attractiveness of distant opinions: $\alpha>1$ suggests that distant opinions are more attractive, as the marginal cost increases with opinion distance. For $\alpha = 2$, the best-response dynamics coincide with the classic DeGroot model\cite{JRPF:56,MHDG:74}; On the other hand, $\alpha<1$ indicates that agents are more attracted by nearby opinions; $\alpha=1$ neutralizes the effect of opinion distance on opinion attractiveness and has been theoretically analyzed in \cite{WM-JMH-GC-FB-FD:24}.

In this paper, we provide both theoretical analysis and simulation studies on this class of opinion dynamics models based on individuals' best responses to social pressure. Specifically, we investigate how different values of $\alpha$ influence the convergence properties and consensus conditions of the system.

\subsection{Literature review}
 Opinion dynamics focus on modeling how interpersonal influences lead to global patterns such as consensus, polarization, or fragmentation. Early models are primarily based on rules of thumb, directly imposing rules on how individuals update their opinions. For example, the Friedkin-Johnsen (F-J) model introduces prejudice to explain the dissensus of opinions\cite{NEF-ECJ:90}; the Hegselmann-Krause (H-K) model assumes that individuals only assign weights to opinions within certain distances from their own opinions\cite{RH-UK:02,CB-CA-AP-FV:24}; the Altafini model considered the presence of negative weights in the influence network\cite{CA:13}; multi-topic opinion dynamics models consider the interdependence between different topics\cite{FNE-PAV-TR-PSE:16,BA-FA-LNE:23}. For more comprehensive and advanced surveys on opinion dynamics models, please refer to \cite{AVP-RT:17,AVP-RT:18} and \cite{BDOA-MY:19,YT-LW:23}.

Over the past decade, a game-theoretic perspective has emerged in modeling opinion evolution, offering a more interpretable and falsifiable way to characterize individuals’ behavioral motivations\cite{VJ-VVS-MI:21,ESR-BT:15,BD-TH-BT:16}. Under the game-theoretic framework, the DeGroot model can be interpreted as individuals’ best responses to the social pressure caused by disagreeing with others\cite{PG-JL-FS:14,DF-CV:19}. Bindel et al. analyzed a variant of the DeGroot model, i.e., the F-J model, and studied the efficiency of its Nash equilibrium in minimizing the total social pressure, namely, the price of anarchy (PoA) \cite{DB-JK-SO:15}.

Mei et al. \cite{WM-FB-GC-JH-FD:22} generalize the DeGroot model to a broader class of best-response dynamics driven by social pressure. In their paper, the authors point out that the reason why the DeGroot model always reaches consensus under very mild connectivity conditions lies in the fact that the attractiveness of opinions increases with opinion distance. In their model, the relationship between opinion distance and social influence is governed by a parameter $\alpha$. In \cite{WM-JMH-GC-FB-FD:24}, only the case $\alpha = 1$, i.e., the weighted-median opinion dynamic, is thoroughly analyzed, while the cases for other values of $\alpha$ remain unexplored.

\subsection{State of contribution}
In this paper, we study the class of best-response opinion dynamics models proposed in \cite{WM-FB-GC-JH-FD:22}, and analyze the impact of the exponent parameter $\alpha$, which determines the relationship between opinion distance and attractiveness, on the convergence and consensus properties of the system. Through theoretical analysis, we show that when $\alpha > 1$, the system exhibits similar convergence and consensus behavior to that of the DeGroot model. That is, we establish the convergence of dynamics and prove that strongly connectivity indicates consensus.  For $\alpha < 1$, we show that individuals’ best-response behavior resembles that in the $ \alpha = 1$ case (the weighted median model), in the sense that agents always choose their opinions from among the existing ones. We construct a counterexample to show that the dynamics does not always converge in this case, and we derive sufficient conditions under which convergence and consensus are guaranteed. Furthermore, our simulation results on small-world networks reveal that when the network structure is moderately random, increasing $\alpha$ significantly reduces the probability of consensus and increases opinion diversity.
However, as the rewiring probability increases, the system becomes more likely to achieve consensus, which in turn reduces opinion diversity and diminishes its sensitivity to $\alpha$.  Moreover, it is shown in simulations that the inertia coefficient $\beta$ has little effect on whether consensus is achieved.

\subsection{Organizations}
The remainder of this paper is structured as follows: Section II presents the mathematical formulation of our opinion dynamics model. Section III establishes theoretical results on convergence and consensus conditions for $\alpha>1$. Section IV analyses the case when $\alpha<1$. Section V provides numerical simulations illustrating the role of \( \alpha \) in shaping opinion evolution. Section VI concludes.

\section{Basic Definitions And Model Setup}

\subsection{Notations and Definitions}
Let $\subseteq$ ($\subset$ resp.) denote subset (proper subset resp.). Denote by $\mathbb{N}$ the set of natural numbers, i.e., $\mathbb{N}=\{0,1,2,\dots\}$. Let $\mathbb{Z}$ ($\mathbb{Z}_+$ resp.) be the set of (positive resp.) integers. Let $\mathbb{R}^n$ be the n-dimensional Euclidean space. Let $\mathbf{1}_n$ ($\mathbf{0}_n$ resp.) be the $n$-dimension vector whose entries are all ones (zeros resp.). For any two vectors $x,y\in \mathbb{R}^n$, let $x\leq y$ represent that $x_i\leq y_i$ for all $1\leq i\leq n$ and let $x\lneq y$ represent that $x\leq y$ and there exists $ 1\leq j\leq n$ such that $x_j<y_j$.

Consider a group of $n$ individuals labelled by $\V = \{1,\dots,n\}$. In this paper, we use the term individual, agent and node interchangably. The interaction relationships among the group are described by a directed and weighted graph $\G(W)$, where $W=(w_{ij})_{n\times n}$ is the associated adjacency matrix. Here, $w_{ij} > 0$ means that there is a link from node $i$ to node $j$ with weight $w_{ij}$, or, equivalently, individual $i$ assigns $w_{ij} > 0$ weight to individual $j$. Since nodes in $\G(W)$ represent individuals in opinion dynamics, we also call $W$ \emph{influence matrix}. Conventionally, an influence matrix $W$ is assumed to be row-stochastic.
In a digraph $\mathcal{G}$, every node with out-degree 0 is called a sink. A subgraph $\mathcal{H}$ is called a strongly connected component of $\mathcal{G}$ if $\mathcal{H}$ is strongly connected and any other subgraph of $\mathcal{G}$ strictly containing $\mathcal{H}$ is not strongly connected. The condensation digraph of $\mathcal{G}$, denoted by $\mathcal{C}(\mathcal{G})$, is defined as follows: the nodes of $\mathcal{C}(\mathcal{G})$ are the strongly connected components of $\mathcal{G}$, and there exists a directed edge in $\mathcal{C}(\mathcal{G})$ from node $\mathcal{H}_1$ to $\mathcal{H}_2$ if and only if there exists a directed edge in $\mathcal{G}$ from a node of $\mathcal{H}_1$ to a node of $\mathcal{H}_2$. 

\subsection{The best-response opinion dynamics (BROD)}

In this section, we present our game-theoretic opinion dynamics model. The central idea is that each agent updates its opinion towards the value that minimizes its perceived social cost. For each agent \(i\), given the current opinion vector \(x = (x_1, x_2, \dots, x_n)\), we define the social cost of choosing a candidate opinion \(z \in \mathbb{R}\) by
\[
u_i(z;x) = \sum_{j=1}^n w_{ij}\, |z - x_j|^\alpha,
\]
where \(\alpha > 0\) determines the sensitivity to opinion differences. To achieve a best response, the ideal update for agent \(i\) is obtained by solving
\begin{equation}\label{eq:opt-problem}
\min_{z \in \mathbb{R}} u_i(z;x).    
\end{equation}

A potential issue is that the optimization problem may yield a set of minimizers rather than a unique value when \(\alpha \leq 1\). To ensure that the update rule is well-defined, we introduce the following mapping of best response:

\begin{definition}[Best-Response Operator]\label{def:br-operator}
For any given opinion vector $x \in \mathbb{R}^n$, weight matrix $W \in \mathbb{R}^{n\times n} $, and parameter $\alpha > 0$, the best-response operator $\operatorname{BR}^\alpha(x,W)$ is defined as a mapping from $\mathbb{R}^n\times \mathbb{R}^{n\times n}$ to $\mathbb{R}^n$. In particular, let $\operatorname{BR}^\alpha_i(x,W)$ denote the $i$th component of this mapping and then $\operatorname{BR}^\alpha_i(x,W)$ is defined to satisfy the following three rules:
\begin{enumerate}
    \item Cost Minimization:\\
    $\operatorname{BR}^\alpha_i(x,W) \in \arg\min_{z \in \mathbb{R}} u_i(z;x).$
    \item Proximity Criterion:  \\
    $\vert \operatorname{BR}^\alpha_i(x,W)-x_i| \leq |z^*-x_i|,\quad \forall z^*\in\arg\min_{z \in \mathbb{R}} u_i(z;x)$
    \item Deterministic Tie-Breaking:
    If there exist distinct $z_1,z_2 \in \arg\min_{z \in \mathbb{R}} u_i(z;x)$ with 
    \[
    |z_1 - x_i| = |z_2 - x_i| = \min_{z \in \arg\min_{z \in \mathbb{R}} u_i(z;x)} |z - x_i|,
    \]
    then $\operatorname{BR}^\alpha_i(x,W)=\min \{z_1,z_2\}$
\end{enumerate}
\end{definition}

With the best-response operator \(\operatorname{BR}^\alpha(x;W)\) providing a well-defined opinion update rule for each agent minimizing his cost, we now formalize the overall opinion dynamics. In particular, we consider a synchronous update rule where all agents adjust their opinions simultaneously by incorporating both the best response and a fixed level of inertia.

\begin{definition}[Best-Response Opinion Dynamics]\label{def:opinion-dynamic}
Consider a group of $n$ individuals in an influence network associated with a row-stochastic influence matrix $W$. For any $i \in  \V$ and any $k \in \mathbb{N} $, denote by $x_i(k)$ individual $i$’s opinion at time $k$. The
best-response opinion dynamics (BROD) updates individual's opinions synchronously at each time $k$ according to
the following equation:
\begin{equation}\label{eq:dynamic}
    x(k + 1) = \beta x(k)+(1-\beta)\BR^\alpha(x(k);W),\forall i=1,\cdots,n
\end{equation}
where $\beta\in(0,1)$ is the coefficient of inertia term.
\end{definition}

The inclusion of the inertia term \(\beta\, x_i(k)\) serves two primary purposes:
\begin{enumerate}
    \item Capturing Realistic Dynamics: In real-world social systems, opinions typically evolve gradually rather than through abrupt changes. The inertia term ensures that an agent does not completely abandon its current opinion, reflecting the observed resistance to sudden changes.
    \item Promoting Diversity and Compromise: Especially when \(\alpha \leq 1\), cost minimization often leads agents to adopt existing opinions, inhibiting the generation of new opinions as discussed in Lemma~\ref{lem:no-new-opinion}. By blending the best response with the agent’s previous opinion, the inertia term facilitates smoother transitions, thereby promoting the emergence of compromise and opinion diversity.
\end{enumerate}

This model setup forms the foundation for our subsequent analysis of convergence and consensus properties, which we develop in the following sections.

\section{Analysis for the case $\alpha>1$}

In this section, we establish the convergence of BROD for $\alpha>1$ and derive a necessary and sufficient condition for the system to reach consensus under any initial states. 

\subsection{Convergence}

In order to prove the convergence of BROD, we first introduce the following definitions and lemma.

\begin{definition}[Positive function] A map $f:\mathbb{R}^n\rightarrow\mathbb{R}^n$ is said to be positive if $f(x)\in\mathbb{R}_{\ge 0}^n$ for all $x\in\mathbb{R}_{\ge 0}^n$.
\end{definition}

\begin{definition}[Type-K order-preserving] A positive map $f:\mathbb{R}^n\rightarrow\mathbb{R}^n$ is said to be type-K order-preserving if for any $x,y\in\mathbb{R}_{\ge 0}^n$ and $x\lneq y$, it holds
\begin{align*}
    (i)&\  x_i = y_i \Rightarrow f_i(x)\leq f_i(y)\\
    (ii)&\  x_i<y_i \Rightarrow f_i(x)< f_i(y)
\end{align*}
for all $i = 1,2,...,n$, where $x\lneq y \Leftrightarrow x \leq y \ \text{and} \ x \neq y$. 
\end{definition}

\begin{definition}[Sub-homogeneous] A positive map $f:\mathbb{R}^n\rightarrow\mathbb{R}^n$ is said to be sub-homogeneous if $kf(x)\leq f(kx)$, for all $x\in \mathbb{R}_{\ge 0}^n$ and $k\in [0,1]$.
    
\end{definition}

\begin{lemma}[\cite{deplano2020nonlinear},Theorem 13]\label{Lemma:cvg}
    Let a positive continuous map $f:\mathbb{R}^n\rightarrow\mathbb{R}^n$ be type-K order-preserving and sub-homogeneous. If $f$ has at least one positive fixed point in the interior of $\mathbb{R}_{\ge 0}^n$ then all periodic points are fixed points, which means 
    \begin{align*}
        \lim_{k\rightarrow \infty}f^k(x)=\bar{x},\ \ \ \forall x\in\mathbb{R}_{\ge 0}^n
    \end{align*}
    where $\bar{x}$ is a fixed point of $f$.
\end{lemma}

The following theorem gives a proof of the convergence of BROD.

\begin{theorem}\label{Prop:cvg}
    Consider the best-response opinion dynamics given by Definition \ref{def:opinion-dynamic} on an influence network $\mathcal{G}(W)$ for $\alpha > 1$. For any initial state $x(0)\in \mathbb{R}^n$, the solution $x(k)$ always converges to an equilibrium.
\end{theorem}

\begin{proof}
    According to Lemma \ref{Lemma:cvg}, it is sufficient to prove that the update rule $f(x) \coloneqq \beta x(k)+(1-\beta)\BR^\alpha(x(k);W)$ satisfies the following five properties when $\alpha>1$: continuous, positive, type-K order-preserving, sub-homogeneous, and has at least one positive fixed point.

    First, we prove that $f$ is continuous. Since the absolute value function is a strictly convex function, $\sum_{j=1}^nw_{ij}|z-kx_j|^\alpha$ is a strictly convex function. It is observed that $BR^\alpha_i(x,W)\in[\min\limits_k(x_k),\max\limits_k(x_k)]$ for any $1\leq i\leq n$. Therefore, $\mathop{\operatorname{argmin}}\limits_z(\sum_{j=1}^nw_{ij}|z-x_j|^\alpha)$ is equal to $\mathop{\operatorname{argmin}}\limits_{z\in[\min\limits_k(x_k),\max\limits_k(x_k)]}(\sum_{j=1}^nw_{ij}|z-x_j|^\alpha)$. The Theorem 9.17 in \cite{sundaram1996first} therefore guarantees the continuous of $\BR^\alpha(x(k);W)$. As a result, $f$ is continuous when $\alpha>1$.
    
    Next, we show that $f$ is positive, which means for all $x\in\mathbb{R}_{\ge 0}^n$, $f(x)\in\mathbb{R}_{\ge 0}^n$. For any $1\leq i \leq n$,
    \begin{align*}
        f_i(x)= (1-\beta)\BR_i^\alpha(x;W)+\beta x_i
    \end{align*}
    For the first part, $\mathop{\operatorname{argmin}}\limits_x(\sum_{j=1}^nw_{ij}|x-x_i|^\alpha)\ge \mathop{\operatorname{min}}\limits_i (x_i)\ge 0$. For the second part, $x_i\ge 0$. Therefore, $f_i(x)\ge 0$, which implies that $f(x)\ge 0$ for all $x \geq 0$.

    In order to prove that $f$ is type-K order-preserving, we first prove the following conclusion: For any $1\leq i \leq n$, if $x,y\in\mathbb{R}_{\ge 0}^n$ and $x_i\neq y_i,\ x_j = y_j$ for all $j\neq i$, then $f_i(x) < f_i(y)$,\ $f_j(x)\leq f_j(y)$, for all $j\neq i$.     
    Since the sum of absolute value functions is a strictly convex function, therefore when $\alpha>1$, $\mathop{\operatorname{argmin}}\limits_z(u_k(z;x))$ is the unique solution of equation $\frac{du_k(z;x)}{dz} = 0$ for any $1\leq k \leq n$. Denote $x^*_k = \mathop{\operatorname{argmin}}\limits_z(u_k(z;x)),y^*_k=\mathop{\operatorname{argmin}}\limits_z(u_k(z;y))$, we have
    $
        \frac{d(\sum_{j=1}^nw_{kj}|z-x_j|^\alpha)}{dz}|_{z=x^*_k}=0,\ 
        \frac{d(\sum_{j=1}^nw_{kj}|z-y_j|^\alpha)}{dz}|_{z=y^*_k}=0.
    $
    For all $1\leq k\leq n$,
    \begin{align*}
        \frac{du_k(z;x)}{dz}&=\frac{d(\sum_{j=1}^nw_{kj}|z-x_i|^\alpha)}{dz}\\
        &=\frac{d(\sum_{j=1}^nw_{kj}((z-x_j)^2)^\frac{\alpha}{2}}{dz}\\
        &=2\sum_{j=1}^nw_{kj}((z-x_j)^2)^{\frac{\alpha}{2}-1}(z-x_j).
    \end{align*}
    Define $g(x)=x|x|^{\alpha-2}$. Since $g(x)$ is an increasing function when $\alpha>1$, we have
    \begin{align*}
        \frac{1}{2}\frac{du_k(z;y)}{dz}|_{z=x_k^*}=&\sum_{j=1}^nw_{kj}((x_k^*-y_j)^2)^{\frac{\alpha}{2}-1}(x_k^*-y_j)\\
        =&\sum_{j\neq i}w_{kj}((x_k^*-x_j)^2)^{\frac{\alpha}{2}-1}(x_k^*-x_j)\\
        &+w_{ki}((x_k^*-y_i)^2)^{\frac{\alpha}{2}-1}(x_k^*-y_i)\\
        =&-w_{ki}((x_k^*-x_i)^2)^{\frac{\alpha}{2}-1}(x_k^*-x_i))\\
        &+w_{ki}((x_k^*-y_i)^2)^{\frac{\alpha}{2}-1}(x_k^*-y_i))\\
        =&w_{ki}\left(g(x_k^*-y_i)-g(x_k^*-x_i)\right).
    \end{align*}
    Since $x_k^*-y_i<x_k^*-x_i$, therefore $\frac{du_k(z;y)}{dz}|_{z=x_k^*}\leq 0$, which implies that $x_k^*\leq y_k^*$, i.e.,  $\mathop{\operatorname{argmin}}\limits_z(u_k(z;x))\leq \mathop{\operatorname{argmin}}\limits_z(u_k(z;y))$ for all $1\leq k\leq n$. 
    Therefore, $f_k(x) = (1-\beta)\mathop{\operatorname{argmin}}\limits_z(u_k(z;x))+\beta x_k\leq (1-\beta)\mathop{\operatorname{argmin}}\limits_z(u_k(z;y))+\beta y_k = f_k(y)$ for all $1\leq k\leq n$. And since $x_i< y_i$, we have $f_i(x)<f_i(y)$.
    
    With the conclusion above, we may now prove that $f$ is type-K order-preserving. For any $x,y\in\mathbb{R}_{\ge 0}^n$ and $x\lneq y$, construct a vector sequence $\{x^{k}\}$ which satisfies $x^0=x,x^{n}=y$, and the first $n-k$ components of $x^{k}$ are the same as the first $k$ components of $x$, the last $k$ components of $x^{k}$ are the same as the last $n-k$ components of $y$, for all $0 \leq k \leq n$. Therefore we have $f(x)=f(x^0)\leq f(x^1)\leq\cdots\leq f(x^n)=f(y)$, and if $x_i<y_i$ for some $1\leq i\leq n$, $f_i(x^{(n-i)})<f_i(x^{(n-i+1)})$. Therefore, $f$ is a type-K order-preserving function.

    Then, we show that $f$ is sub-homogeneous. For all $x\in \mathbb{R}_{\ge 0}^n$, $1\leq i\leq n$ and $k\in [0,1]$,
    \begin{align*}
    f_i(kx) &= (1-\beta) \mathop{\operatorname{argmin}}_z(\sum_{j=1}^nw_{ij}|z-kx_j|^\alpha)+\beta kx_i \\
    &= (1-\beta)k^\alpha \mathop{\operatorname{argmin}}_z(\sum_{j=1}^nw_{ij}|z-x_j|^\alpha)+\beta kx_i\\
    &\ge k((1-\beta) \mathop{\operatorname{argmin}}_z(\sum_{j=1}^nw_{ij}|z-x_j|^\alpha)+\beta x_i)\\
    &=kf_i(x),
    \end{align*}
    which means that $f(kx)\ge kf(x)$.
    
    Finally, since $f(c\mathbf{1}_n)=c\mathbf{1}_n$ for all $c>0$, with all five properties satisfied, the convergence of the system is thus guaranteed by Lemma \ref{Lemma:cvg}.
\end{proof}

\subsection{Consensus}
In this section, we provide a necessary and sufficient condition for BROD to reach consensus. 

\begin{theorem}\label{Prop:css} Consider the best-response opinion dynamics given by Definition \ref{def:opinion-dynamic} on an influence network $\mathcal{G}(W)$ for $\alpha > 1$. The solution $x(k)$ converges to a consensus equilibrium for any initial state $x(0) \in \mathbb{R}^n$ if and only if graph $\mathcal{G}$ has a unique globally-reachable strongly-connected component, i.e., the condensation digraph $\mathcal{C}(\mathcal{G})$ contains a unique (globally-reachable) sink.
\end{theorem}

\begin{proof}
For the simplicity of terminology, we denote the length of the longest path from node $i$ to node $j$ in an acyclic digraph as node $i$'s ``major length'' to node $j$, analogous to the notion of the "major arc" on a circle. Since the convergence of the system has been guaranteed by Theorem \ref{Prop:cvg}, we may define $x^*=(x_1^*,x_2^*,\cdots,x_n^*)=\mathop{\lim}\limits_{k\rightarrow\infty}x(k)$.

Regarding the necessary part, partition graph $\mathcal{G}$ into several disjoint strongly-connected components $\mathcal{G}_1,\mathcal{G}_2,\cdots,\mathcal{G}_s$, where $\mathcal{G}_1$ is the only globally-reachable one. Since $\mathcal{G}_1,\mathcal{G}_2,\cdots,\mathcal{G}_s$ are regarded as nodes in the corresponding condensation digraph $\mathcal{C}(\mathcal{G})$, by arranging them in ascending order based on their major length to the unique sink $\mathcal{G}_1$, we can inductively prove that the system will reach consensus.

Start with the globally-reachable component $\mathcal{G}_1$. Denote $M=\mathop{\max}\limits_{i\in\mathcal{V}_1}(x_i^*)$, $m=\mathop{\min}\limits_{i\in\mathcal{V}_1}(x_i^*)$. 
Proof by contradiction. Assume that $m<M$, then there exists $i_\infty, j_0 \in \mathcal{V}_1$ such that $w_{i_\infty j_0}>0$ and $\mathop{\lim}\limits_{k\rightarrow\infty}x_{i_\infty}(k)=M$, $\mathop{\lim}\limits_{k\rightarrow\infty}x_{j_0}(k)<M$. Since $x_j^*\leq M$ for all $j \in \mathcal{V}_1$ and $x_{j_0}^*<M$, we have 
\begin{align*}
    \mathop{\operatorname{argmin}}_z(\sum_{j=1}^n w_{i_\infty j}|z-x_j^*|^\alpha)<M,
\end{align*}
Therefore, 
\begin{align*}
    f_{i_\infty}(x^*)<M = x_{i_\infty}^*,
\end{align*}
which leads to a contradiction since $x^*$ should be an equilibrium point of $f(x)$. Therefore, the opinions of nodes in $\mathcal{G}_1$ (strongly-connected components whose major length to sink $\mathcal{G}_1$ in $\mathcal{C}(\mathcal{G})$ is 0) will converge to a consensus equilibrium for any initial state $x(0)$. Denote by $\hat{x}_0 \in \mathbb{R}$ the consensus opinion of nodes in $\mathcal{G}_1$.

Then suppose that the opinions of nodes in any strongly-connected component whose major length to $\mathcal{G}_1$ in $\mathcal{C}(\mathcal{G})$ is less than $d$ will converge to a consensus equilibrium, for any $d \geq 1,\ d \in \mathbb{Z}$. Let $\mathcal{G}_{j_d}$ be any strongly-connected component whose major length to $\mathcal{G}_1$ in $\mathcal{C}(\mathcal{G})$ is $d$. Similarly, denote $M_{\mathcal{G}_{j_d}}=\max\limits_{i\in\mathcal{V}_{j_d}}(x_i^*), m_{\mathcal{G}_{j_d}}=\min\limits_{i\in\mathcal{V}_{j_d}}(x_i^*)$. 
Since that nodes in $\mathcal{G}_{j_d}$ interact only among themselves and with nodes from components whose major length to $\mathcal{G}_1$ is less than $d$, i.e., with nodes whose opinions will converge to $\hat{x}_0$. Therefore, by applying the inductive hypothesis and employing a contradiction argument similar to the proof on $\mathcal{G}_1$, we can prove that $M_{\mathcal{G}_{j_d}}\leq \hat{x}_0$ and $m_{\mathcal{G}_{j_d}}\ge \hat{x}_0$, which means $M_{\mathcal{G}_{j_d}}=m_{\mathcal{G}_{j_d}}=\hat{x}_0$. Therefore, the opinions of nodes in $\mathcal{G}_{j_d}$ will converge to a consensus equilibrium for any initial state $x(0)$, thereby concluding the inductive proof.

Regarding the sufficient part, assume that the dynamic process converges to consensus. Proof by contradiction. If there exist more than one globally-reachable strongly-connected components, that is to say, the condensation digraph $\mathcal{C}(\mathcal{G})$ contains more than two sinks, set the initial opinions such that the nodes in the same sink share the same opinions, while nodes in different sinks hold distinct opinions. Under this condition, the opinions of nodes in each globally-reachable strongly-connected component converge, but they do not reach consensus, thereby leads to a contradiction. Therefore, graph $\mathcal{G}$ can only have one unique globally-reachable strongly-connected component.

\end{proof}

\section{Analysis for the case $\alpha<1$}

In this section, we first show that when \(\alpha < 1\), each agent’s best response is selected among the existing opinions. We then construct a counterexample to demonstrate that the system do not always converge. Finally, we provide a sufficient condition under which the system achieves convergence and consensus.

\subsection{Property of the Best-Response Operator}

For $\alpha < 1$, the best-response operator behaves differently from the case \(\alpha > 1\). In particular, for any node $i$, $\BR^\alpha_i(x;W)$ is restricted to the existing values of opinions. The following lemma formalizes this property.
\smallskip

\begin{lemma}\label{lem:no-new-opinion} When $0<\alpha\leq1$, it holds that 
\begin{align*}
    \BR^\alpha_i(x;W)\in \{x_j\vert j=1,\cdots,n\},\forall i=1,\cdots,n.
\end{align*}
\end{lemma}  
\smallskip
    
\begin{proof}
    To analyze the optimization problem in Equation~\ref{eq:opt-problem}, consider the states $x$ sorted as:
\[
x^{(1)} \leq x^{(2)} \leq \cdots \leq x^{(n)}.
\]
The domain of the objective function can then be divided into intervals \([x^{(i)}, x^{(i+1)}]\) for \(i = 1, \ldots, n-1\).

Within each interval, the absolute value terms \(|z - x_j(k)|^\alpha\) can be simplified to remove the absolute value sign, leading to a piecewise function. The first derivative of the objective function in an interval is:
\[
f'(z) = \alpha \sum_{j: x_j \leq z} w_{ij} (z - x_j)^{\alpha - 1} - \alpha \sum_{j: x_j > z} w_{ij} (x_j - z)^{\alpha - 1}.
\]
The second derivative is:
\begin{align*}
\frac{1}{\alpha (\alpha - 1)}f''(z) =  & \sum_{j: x_j \leq z} w_{ij} (z - x_j(k))^{\alpha - 2} \\
&+ \sum_{j: x_j > z} w_{ij} (x_j(k) - z)^{\alpha - 2}.
\end{align*}
Since \(0 < \alpha < 1\), we have \(\alpha - 1 < 0\), implying \(f''(z) < 0\). Thus, \(f(z)\) is concave within each interval.

Given that \(f(z)\) is concave within each interval, the minimum of \(f(z)\) must occur at the boundaries of the interval. Therefore, the optimal solution $\BR^\alpha_i(x;W)$ is always one of the elements in the set $\{x_j\vert j=1,\cdots,n\}$.

\end{proof}

\subsection{Sufficient Conditions for Convergence and Consensus}
In this section we propose sufficient conditions for the convergence and consensus of BROD when $\alpha<1$. Before proceeding, we first provide a counterexample where the system does not converge, despite the presence of inertia.

\begin{example}\label{example:not-converge}
Consider a system with \( n = 6 \) agents and 3 distinct opinions. The initial opinion vector is $x(0) = (0, 3, 3 , 0 , 2 , 3) $ and the interaction weight matrix \( W \) is given by:
\[
W = \begin{bmatrix}
0 & 0 & 0.1 & 0.4 & 0.3 & 0.2 \\
0.2 & 0 & 0 & 0.1 & 0.25 & 0.45 \\
0 & 0.05 & 0 & 0.45 & 0.1 & 0.4 \\
0 & 0 & 0 & 1 & 0 & 0 \\
0 & 0 & 0 & 0 & 1 & 0 \\
0 & 0 & 0 & 0 & 0 & 1
\end{bmatrix}
\]

\begin{figure}[thpb]
    \centering
    \includegraphics[width=0.8\linewidth]{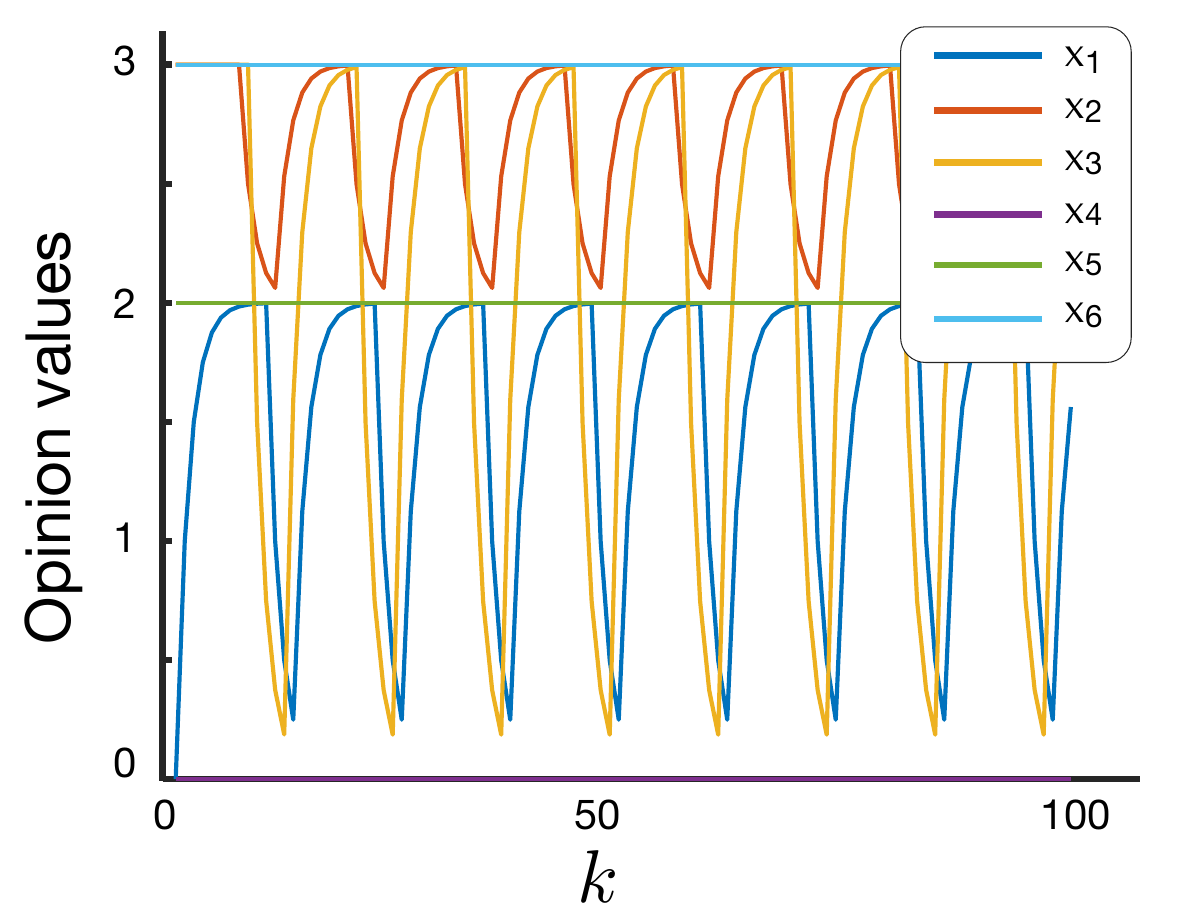}
    \caption{The system describled in Example~\ref{example:not-converge} oscillates at $\alpha=0.5,\beta=0.4$.   }
    \label{fig:counterexample}
\end{figure}

In this case, the system fails to converge at $\alpha=0.5,\beta=0.4$, as illustrated in Fig \ref{fig:counterexample}. Actually, for the influence matrix $W$ given in the example, adding a sufficiently small  $\epsilon$  to each element and then performing row normalization so that $\mathcal{G}(W)$ becomes a complete graph could still result in non-convergent oscillatory behaviors (could still induce persistent oscillations in opinions). This suggests that the decisive structure of interactions cannot be fully captured by the influence matrix $W$ alone when $\alpha<1$. 
\end{example}

While we provide a counterexample, characterizing necessary and sufficient conditions for convergence when $\alpha<1$ remains a challenging problem. Compared to the case $\alpha = 1$ (i.e., the weighted median rule), where the best response depends only on the relative ordering of opinions and the network weights, the case $\alpha < 1$ introduces additional complexity: the best response is determined not only by the ranking of opinions, but also by the specific values of opinions and weights. These dependencies significantly complicate theoretical analysis and make the convergence behavior more intricate. Developing a deeper understanding of this structure is an interesting and worthwhile direction for future research.

Now we provide a sufficient condition for convergence when $\alpha < 1$. If for all $1\leq i\leq n$, there exists a $1\leq j_i\leq n$ such that $w_{ij_i}>\frac{1}{2}$, then the system converges.

\begin{theorem}\label{thm:convergence}
Consider the best-response opinion dynamics given by Definition \ref{def:opinion-dynamic} on an influence network $\mathcal{G}(W)$ for \(\alpha \in (0,1)\). Suppose that for each agent \(i\in\{1,2,\dots,n\}\) there exists at least one index \(j\) such that 
\[
w_{ij} > 0.5.
\]
Then for any initial state \(x(0) \in \mathbb{R}^n\), the solution $x(k)$ converges to an equilibrium. 
\end{theorem}

\begin{proof}
Fix \(i\), and denote the corresponding index as \(j\) such that \(w_{ij}>0.5\). For any given \(x\in\mathbb{R}^n\), consider the cost function
\[
u_i(z;x)=\sum_{t=1}^n w_{it}|z-x|^\alpha,
\]
with \(0<\alpha<1\). We claim that \(x_j\) is the unique minimizer of \(u_i(z;x)\); that is, 
\[
\operatorname*{argmin}_{z\in\mathbb{R}} u_i(z;x)=\{x_j\}.
\]
For any \(k\neq j\), note that
\begin{align*}
u_i(x_j;x)-u_i(x_k;x)=&\sum_{t\neq j,k}w_{it}\Bigl(|x_j-x_t|^\alpha-|x_k-x_t|^\alpha\Bigr)\\
&+w_{ik}|x_j-x_k|^\alpha-w_{ij}|x_k-x_j|^\alpha.
\end{align*}
Since for any \(x_t\) we have
\[
\Bigl||x_j-x_t|^\alpha-|x_k-x_t|^\alpha\Bigr|\le|x_j-x_k|^\alpha,
\]
it follows that
\[
u_i(x_j;x)-u_i(x_k;x)\le \Bigl(\sum_{t\neq j,k}w_{it}+w_{ik}-w_{ij}\Bigr)|x_j-x_k|^\alpha.
\]
As \(\sum_{t\neq j,k}w_{it}+w_{ik}=1-w_{ij}\), we have
\[
u_i(x_j)-u_i(x_k)\le (1-2w_{ij})|x_j-x_k|^\alpha.
\]
Since \(w_{ij}>0.5\) implies \(1-2w_{ij}<0\), we conclude that \(u_i(x_j;x)<u_i(x_k;x)\) for all \(k\neq j\). Together with the fact that $\operatorname{BR}_i^\alpha (x,W)\in \{x_k, k=1,\cdots,n\}$ in Lemma~\ref{lem:no-new-opinion}, we can conclude that $\operatorname{BR}_i^\alpha (x,W)=x_j$.

Therefore, the system can be rewritten into the form of 
\[
x(k+1) = f(x(k)) = F\, x(k),
\]
where $F = \bigl[f_{ij}\bigr]_{n\times n}$ is a row-stochastic matrix satisfying $f_{ii} = \beta$ for all $1 \leq i \leq n$, and for each node $i$ there exists only one index $j$ such that $f_{ij} = 1 - \beta$.

Now consider the corresponding digraph $\mathcal{G}(F)$. Notice that each node has only one out-link apart from its self-loop. In this case, each node is either part of a cycle or its only out-link should direct towards a node that belongs to a cycle. According to the Perron-Frobenius theorem, since the adjacency matrix of a cycle with self-loops in graph $\mathcal{G}(F)$ is primitive and row-stochastic, therefore it should have a simple eigenvalue 1 and the norm of all other eigenvalues should be strictly less than 1\cite{FB:21}. Consequently, the opinions of nodes belong to the same cycle converge to a consensus equilibrium. As for the other nodes, they are influenced only by themselves and by a particular node in some cycle, thus their opinions will converge to the consensus opinion of that cycle. The convergence of the system is therefore guaranteed.
\end{proof}

One immediate conclusion from the above result is that under the conditions of Theorem \ref{thm:convergence}, the system $x(k+1) = F\, x(k)$ achieves consensus for any initial state if and only if there exists only one strongly connected component with more than one node in graph $\mathcal{G}(F)$, that is, there exists a unique sequence of distinct nodes $i_1,i_2,\dots,i_l\in\{1,2,\dots,n\}$ such that $w_{i_ki_{k+1}}>\frac{1}{2},\forall 1\leq k\leq l-1$ and $w_{i_li_1}>\frac{1}{2}$.

\section{Simulation Results}

In this section, we provide some simulation results of BROD. Since the case $\alpha<1$ exhibits richer dynamical behaviors, we mainly conduct simulations and analyze the convergence and consensus properties of the system when $\alpha<1$.

We conduct simulations on small-world networks consist of 100 nodes, which can be generated by the classic Watts–Strogatz model\cite{DJW-SHS:98}. For different values of exponent parameter $\alpha \in (0,1)$, inertia coefficient $\beta$, and rewiring probability $p$ of the Watts–Strogatz model, we conduct 100 simulation runs for each parameter setting. We compute the proportion of runs in which the system converges or achieves consensus, as well as the standard deviation of nodes' opinions when the system converges. We refer to the standard deviation of nodes' opinions as the ``opinion diversity". 

\begin{figure}[thpb]
    \centering
    \includegraphics[width=0.9\linewidth]{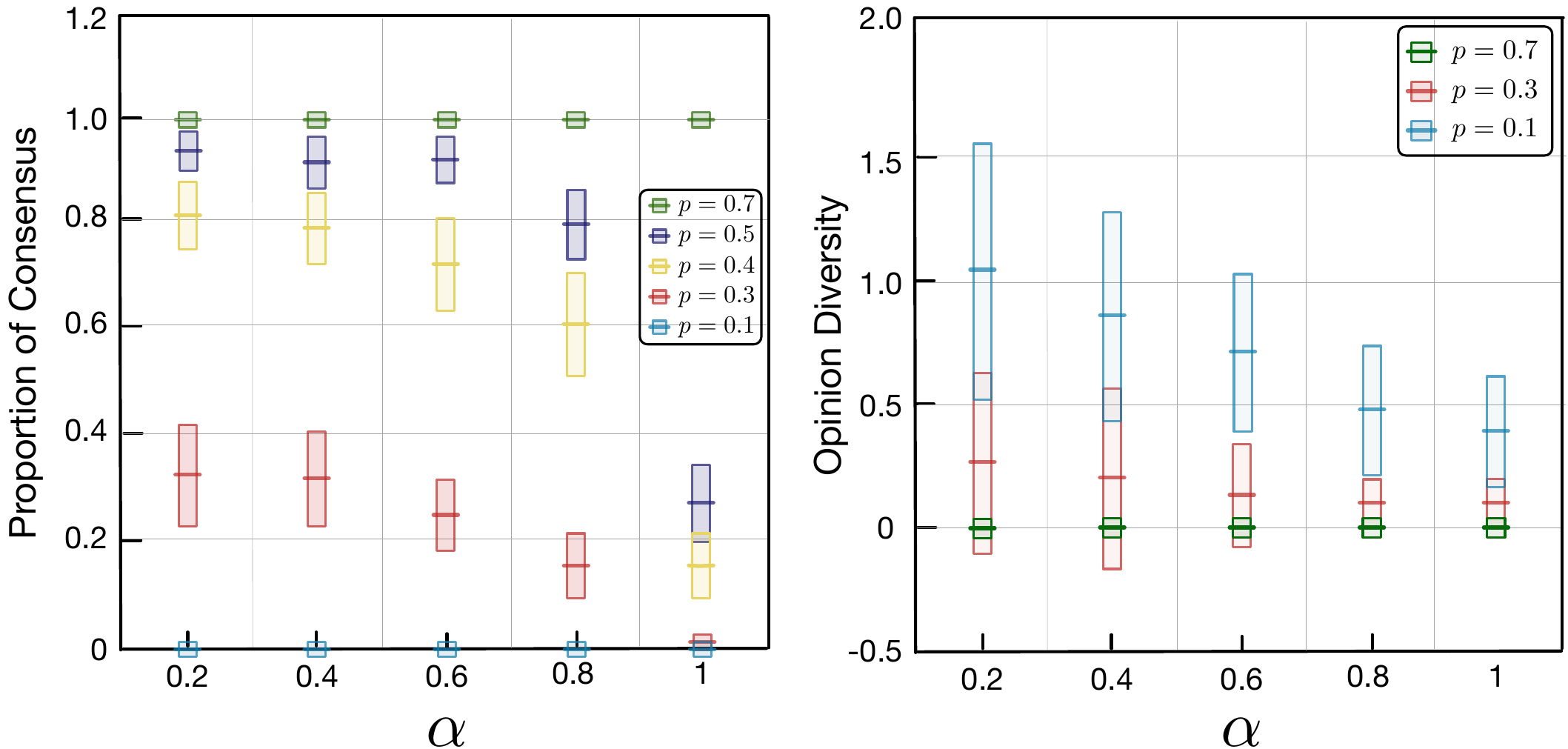}
    \caption{Empirical analysis of simulation results for $\alpha$. Panel (a) shows the relationship between Proportion of Consensus and $\alpha$. The horizontal colored bars indicate the consensus probability, while the vertical ranges of the colored rectangles are the associated 95\% confidence intervals, computed by the \textit{binomial distribution method} \cite{MB:15}. Panel (b) shows the relationship between Opinion Diversity and $\alpha$. The horizontal bars indicate the mean of opinion diversity, while the vertical ranges of the rectangles indicate the standard deviation of opinion diversity. For different values of inertia coefficient $\beta$, the results are qualitatively similar.}
    \label{fig:alpha}
\end{figure}

Simulation results indicate that BROD converges in most cases. Figure \ref{fig:alpha} shows how the consensus probability and opinion diversity vary with $\alpha$, respectively, where different colors represent results under different rewiring probability $p$. It can be seen that when the rewiring probability $p$ is either very large or very small, the probability of the system reaching consensus is weakly correlated with $\alpha$ as it is mainly determined by $p$. While for moderate values of $p$, the consensus probability decreases as $\alpha$ increases. This implies that when the network is highly clustered (small $p$) or highly random (large $p$), $\alpha$ has little influence on whether BROD would achieve consensus. However, when the network structure maintains a balance between clustering and randomness, the influence of $\alpha$ on the consensus property of the system becomes more significant. Regarding opinion diversity, it can also be concluded that when the rewiring probability $p$ takes a relatively small value, the opinion diversity decreases as $\alpha$ increases. As $p$ increases, the probability of the system achieving consensus also increases, further reducing the opinion diversity and therefore weakening the significance of its dependence on $\alpha$.

\begin{figure}[thpb]
    \centering
    \includegraphics[width=0.9\linewidth]{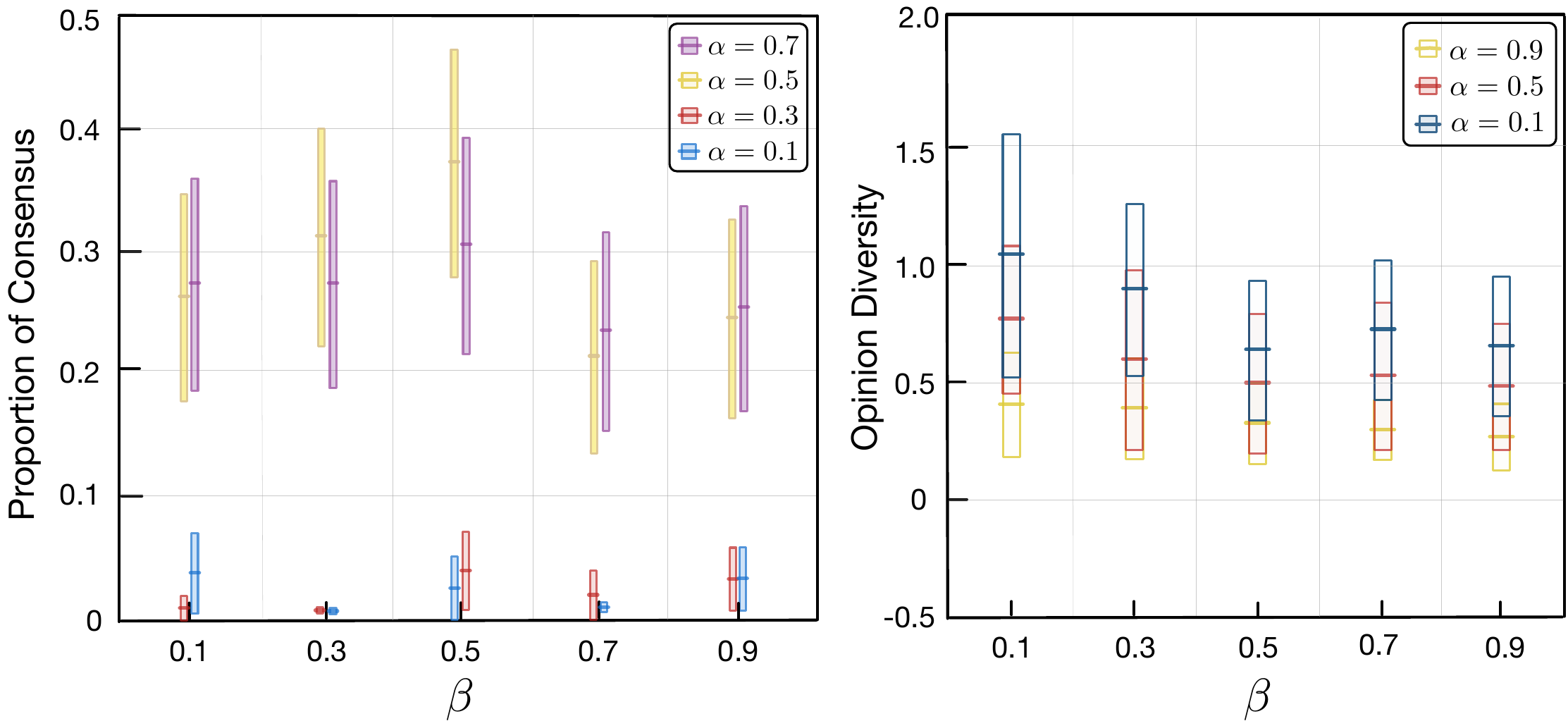}
    \caption{Relationship between Opinion Diversity and $\beta$. The results illustrated in this figure share a similar interpretation with those in Figure \ref{fig:alpha}. For different values of rewiring probability $p$, the results are qualitatively similar.}
    \label{fig:beta}
\end{figure}

Similarly, Figure \ref{fig:beta} shows how consensus probability and opinion diversity vary with $\beta$, respectively, where different colors here represent results under different $\alpha$. It can be concluded that $\beta$ has little influence on the probability of the system achieving consensus. While the opinion diversity slightly decreases as $\beta$ increases, though this change is not statistically significant. As a result, $\alpha$ plays a significant role in shaping the consensus behavior of BROD, while the influence of the inertia coefficient $\beta$ is relatively insignificant.

\section{Conclusions}
In this paper, we focus on a class of best response opinion dynamics parameterized by a tunable exponent \(\alpha\), which controls the relationship between opinion distance and opinion attractiveness. Under different $\alpha$, convergence and consensus properties of the system are studied via theoretical analysis and numerical simulations.    

For the case \(\alpha > 1\), we established the convergence of the proposed opinion dynamics and derived a necessary and sufficient condition for achieving consensus. In contrast, when \(\alpha < 1\), the behavior of the dynamics is considerably more intricate. We demonstrated through a counterexample that convergence is not universally guaranteed. Nevertheless, we provided a sufficient condition under which convergence and consensus can still be ensured.

To further explore the underlying behavior of opinion formation when \(\alpha < 1\), we performed extensive numerical simulations. The results show that the parameter $\alpha$ is an effective factor, dominating the formation of consensus and opinion diversity.  

In future work, it will be meaningful to extend this research framework by incorporating additional aspects, such as heterogeneous agent behaviors, dynamic network structures, or noise and external influences.

\bibliographystyle{IEEEtran}
\bibliography{alias,WM,Main,New,HY_Add}

\end{document}